\documentclass{amsart}
\usepackage{amssymb, url}

\newtheorem{theorem}{Theorem}[section]
\newtheorem{proposition}[theorem]{Proposition}
\newtheorem{lemma}[theorem]{Lemma}
\newtheorem{claim}[theorem]{Claim}
\newtheorem{corollary}[theorem]{Corollary}

\theoremstyle{definition}
\newtheorem{definition}[theorem]{Definition}

\newtheorem{question}[theorem]{Question}

\newcommand{\U}{\mathcal U}

\newcommand{\A}{\mathcal A}
\newcommand{\C}{\mathcal C}
\newcommand{\B}{\mathcal B}
\newcommand{\F}{\mathcal F}
\newcommand{\N}{\mathcal N}
\newcommand{\w}{\omega}

\newcommand{\Sym}{\mathrm{Sym}}
\newcommand{\cf}{\mathrm{cf}}

\newcommand{\la}{\langle}
\newcommand{\ra}{\rangle}

\newcommand{\uhr}{\upharpoonright}

\newcommand{\LL}{\mathcal L}
\newcommand{\hot}{\mathfrak}

\title[Chains of proper subgroups covering a topological group]{On the length of  chains of proper subgroups covering a topological group}

\author[T.~Banakh, D.~Repov\v{s}, and  L.~Zdomskyy]{Taras Banakh, Du\v{s}an Repov\v{s}, and Lyubomyr Zdomskyy}

\address{Department of Mathematics, Ivan Franko National University of Lviv, Ukraine; and\\
Instytut Matematyki, Uniwersytet Humanistyczno-Przyrodniczy Jana
Kochanowskiego w Kielcach, Poland.} \email{tbanakh@yahoo.com}
\urladdr{http://www.franko.lviv.ua/faculty/mechmat/Departments/Topology/bancv.html}

\address{Faculty of Mathematics and Physics, and Faculty of Education, University of Ljubljana, P. O. Box 2964, Ljubljana, Slovenija 1001.}
\email{dusan.repovs@guest.arnes.si}
\urladdr{http://www.fmf.uni-lj.si/\~{}repovs/index.htm}

\address{Kurt G\"odel Research Center for Mathematical Logic,
University of Vienna, W\"ahringer Stra\ss e 25, A-1090 Wien,
Austria.} \email{lzdomsky@gmail.com}
\urladdr{http://www.logic.univie.ac.at/\~{}lzdomsky/}

\subjclass[2000]{Primary:  03E17, 54H11; Secondary:  54D20.}
\keywords{$Q$-point, $P_\kappa$-point, $\sigma$-bounded group,
$\omega$-bounded group, Menger property, $[\F]$-Menger property.}
\thanks{
This research was supported in part by the Slovenian Research Agency
grants P1-0292-0101, J1-2057-0101 and BI-UA/09-10-005. The third
author acknowledges the support of the FWF grant P19898-N18.
We thank the referees for comments and suggestions.}
\date{\today}
\begin{document}

\begin{abstract}
We prove
that if an ultrafilter $\LL$ is
not coherent to a $Q$-point, then
 each  analytic non-$\sigma$-bounded
 topological group $G$ admits an increasing
chain $\la G_\alpha:\alpha<\hot b(\LL)\ra$ of its proper
  subgroups such that: $(i)$\  $\bigcup_{\alpha}G_\alpha=G$; and
  $(ii)$ For every $\sigma$-bounded subgroup $H$ of $G$ there exists $\alpha$ such
 that $H\subset G_\alpha$.
In case of the group $\Sym(\w)$ of all permutations of $\w$
with the topology inherited from $\w^\w$ this
improves upon earlier results  of  S.~Thomas.
\end{abstract}

\maketitle

\section{Introduction}

A  theorem of Macpherson and Neumann \cite{MN90} states that if
the group $\Sym(\w)$ can  be written as a union of an increasing chain $\la
G_i:i<\lambda\ra$ of  proper subgroups $G_i$, then $\lambda>\w$.
Throughout this paper the minimal $\lambda$ with this property will
be denoted by $\cf(\Sym(\w))$. For every increasing function
$f\in\w^\w$ we denote by $S_f$ the subgroup of $\Sym(\w)$ generated
by $\{\pi\in\Sym(\w): \pi,\pi^{-1}\leq^\ast f\}$, where $x\leq^\ast
y$ means that $x(n)\leq y(n)$ for all but finitely many $n\in\w$.
 If
we additionally require that for every $f\in\w^\w$ there exists
$i\in\lambda$ such that $S_f\subset G_i$, then the minimal length of
such a chain will be denoted by $\cf^\ast(\Sym(\w))$. It is clear
that $\cf^\ast(\Sym(\w))\geq\max\{ \cf(\Sym(\w)),\mathfrak b\}$. The
consistency of $\cf^\ast(\Sym(\w))> \cf(\Sym(\w))$ and the
inequality $\cf^\ast(\Sym(\w))\leq \cf(\mathfrak d)$
 were established in \cite[Proposition~2.5]{To98}. The initial aim of this paper
was to sharpen
 the latter upper bound on $\cf^\ast(\Sym(\w))$. This led us to
 consider increasing chains of proper submonoids of  topological monoids.

We recall that
a {\em semigroup\/}
is a set  with a binary
associative operation $\cdot:X\times X\to X$.  A semigroup with a
two-sided unit  $1$ is called a {\em monoid}. It is clear that each
group is a monoid. By a {\em topological monoid} we understand a
monoid $X$  with a topology $\tau$ making
 the binary operation $\cdot:X\times X\to X$ of $X$ continuous.

\begin{definition} \label{de_cfp}

Let $X$ be a topological monoid (resp. group). The minimal length of
an increasing chain $\la X_i:i<\lambda\ra$ of  proper submonoids
(resp. subgroups) $X_i$ of $X$ such that $ X=\bigcup_{i<\lambda}X_i$
and for every $\sigma$-bounded subset $H$ of $X$
 there exists $i\in\lambda$ such that $H\subset X_i$ will be denoted by
$\cf^\ast_m(X)$ (resp. $\cf^\ast_g(X)$).
\end{definition}

We recall that a subset $B$ of a topological monoid $X$ is said to
be \emph{totally bounded,} if for every  open neighborhood $U$ of
the identity $1$ of $X$ there exists a finite subset $F$ of $X$ such
that $X\subset FU\cap UF$.
A subset $B$ is said to be \emph{$\sigma$-bounded,} if it can be
written as a countable union of  totally bounded subsets.  A direct
verification shows that $\cf^\ast(\Sym(\w))$ as defined in
\cite{To98} and $\cf^\ast_g(\Sym(\w))$ in the sense of our
Definition~\ref{de_cfp} coincide.

It is clear that $\cf^\ast_m(X)\le \cf^\ast_g(X)$ for every
topological group $X$. We do not know whether these cardinals can be
different. Probably the most interesting case is the group
$\Sym(\w)$.

Let  $R$ be a relation on  $\w$ and $x,y\in\w^\w$. We denote by
$[x\,R\,y]$ the set $\{n\in\w:x(n)\,R\, y(n)\}$. For an ultrafilter
$\F$ the notation  $x\leq_\F y$ means $[x\leq y]\in\F$. Let
$\mathfrak b(\F)$ be the cofinality of the linearly ordered set
$(\w^\w,\leq_\F)$.

Following \cite{BRZ} we define  a point $x\in X$ of a topological
monoid $X$ to be {\em left balanced\/} (resp. {\em right
balanced\/}) if for every neighborhood $U\subset X$ of the unit $1$
of $X$ there is a neighborhood $V\subset X$ of $1$ such that
$Vx\subset xU$ (resp. $xV\subset Ux$). Observe that $x$ is left
balanced if the left shift $l_x:X\to X$, $l_x:y\mapsto xy$, is open
at $1$. Let $B_{L}$ and $B_R$ denote respectively the sets of all
left and right balanced points of the monoid $X$. A topological
monoid $X$ is defined to be {\em left balanced} (resp. {\em right
balanced}) if $X= B_L\cdot U$ (resp. $X= U\cdot B_R$) for every
neighborhood $U\subset X$ of the unit $1$ in $X$. If a topological
monoid $X$ is both left and right balanced, then we say that $X$ is
{\em balanced}.

We define a topological monoid $X $ to be a Menger
monoid\footnote{In terms of \cite{BRZ} this means that $(X,
\mu_L\wedge\mu_R)$ is a Menger monoid.}, if for every sequence $\la
U_n:n\in\w\ra$ of open neighborhoods of $1$ there exists a sequence
$\la F_n:n\in\w\ra$ of finite subsets of $X$ such that
$X=\bigcup_{n\in\w} F_n U_n\cap U_n F_n$. A topological monoid $X$
is said to be \emph{$\w$-bounded}, if for every neighborhood $U$ of
$1$ there exists a countable $C\subset X$ such that $X=C\cdot U$.

The following two  theorems are the principal results of this paper.

\begin{theorem} \label{main_mon}
Let $X$ be a first countable $\w$-bounded balanced topological
monoid such that one of its finite powers is not
a Menger monoid. Then $\cf^\ast_m(X)\leq \mathfrak b(\LL)$ for every
ultrafilter $\LL$ which is not coherent to any $Q$-point.
\end{theorem}

\begin{theorem} \label{main_gr}
Let $G$ be an $\w$-bounded topological group such that one of its
finite powers  is not
a Menger monoid. Then $\cf^\ast_g(G)\leq \mathfrak b(\LL)$ for every
ultrafilter $\LL$ which is not coherent to any
$Q$-point.
\end{theorem}


Applying \cite[Proposition~7.5]{BRZ} we conclude that the Baire
space  $\w^\w$ with
the operation of composition is a balanced topological monoid,
and  $\sigma$-bounded subsets of this topological monoid are exactly
those which are contained in the $\sigma$-compact subsets of $\w^\w$.
It is easy to see that $\w^\w$ is not a Menger monoid.
Thus we get the following

\begin{corollary} \label{crt}
Let $\LL$ be an ultrafilter coherent to no $Q$-point. Then $\w^\w$
can be written as the union of an increasing chain of its proper
subsets of length $\leq\hot b(\LL)$, each of which is closed under
composition, and such that every $\sigma$-compact subset of $\w^\w$
is contained in one of the elements of this chain.
\end{corollary}

 A metrizable space $X$ is
 said to be
 \emph{analytic}, if it is a continuous image of
$\w^\w$. A topological group $G$ is called \emph{analytic} if such
is the underlying topological space. Theorem~\ref{main_gr} implies
the following:

\begin{corollary} \label{main}
Let $G$ be an analytic group which is not $\sigma$-bounded.
Then $\cf^\ast_g(G)\leq \mathfrak b(\LL)$ for every ultrafilter
$\LL$ which is not coherent to any $Q$-point.
\end{corollary}

 $\Sym(\w)$ is easily seen to be a $G_\delta$-subset of $\w^\w$ and the composition
as well as the inversion are continuous with respect to the topology
inherited from $\w^\w$. Therefore $\Sym(\w)$  with this
topology is a Polish topological group. A direct verification also
shows that it is not $\sigma$-bounded.

\begin{corollary} \label{main_cor0}
$\cf^\ast(\Sym(\w))\leq \mathfrak b(\LL)$ for every ultrafilter
$\LL$ which is not coherent to a $Q$-point.
\end{corollary}

Combined with the following consequence of
\cite[Theorem~2.8]{LafZhu98}, Corollary~\ref{main_cor0} yields the
upper bound for $\cf^\ast(\Sym(\w))$ obtained earlier in
\cite{To98}.

\begin{proposition} \label{notcoh_q}
There exists an ultrafilter $\LL$ which is not coherent to
any $Q$-point and such that $\mathfrak b(\LL)=\cf(\mathfrak d)$.
\end{proposition}

We recall from \cite{Bl86} that ultrafilters $\F$ and $\U$ on $\w$
are said to be \emph{nearly coherent}, if there exists an increasing
sequence
$\la k_n:n\in\w\ra$ of natural numbers such that\\
 $\bigcup_{n\in
I}[k_n,k_{n+1})\in\F$ if and only if $\bigcup_{n\in
I}[k_n,k_{n+1})\in\U$ for every subset
 $I$ of $\w$. In what follows we  shall drop ``near'' and simply say that
two ultrafilters are coherent.
  In other words, $\F$ and $\U$  are  coherent if and
only if $\phi(\F)=\phi(\U)$ for some increasing surjection
$\phi:\w\to\w$. The  coherence relation is an equivalence relation.
NCF is the statement that all ultrafilters are coherent. Its
consistence was established in \cite{BS87}.

An ultrafilter $\LL$ is called:
\begin{itemize}
 \item \emph{a (pseudo-) $P_\kappa$-point}, where $\kappa$ is a cardinal, if for every
$\LL'\in [\LL]^{<\kappa}$ there exists $L\in\LL$ (resp.
$L\in[\w]^\w$) such that $L\subset^\ast L'$ for all $L'\in\LL'$.
$P_{\w_1}$-points are also called $P$-points;
\item \emph{a simple $P_\kappa$-point}, if there exists a sequence $\la L_\alpha:\alpha<\kappa\ra$
of infinite subsets of $\w$ such that $L_\alpha\subset^\ast L_\beta$
for all $\kappa>\alpha>\beta$ and $\LL=\{X\subset\w:L_\alpha\subset
X$ for some $\alpha<\kappa\}$;
\item \emph{a $Q$-point}, if for every increasing surjection $\phi:\w\to\w$
there exists $L\in\LL$ such that $\phi\uhr L$ is injective;
\item \emph{a Ramsey ultrafilter}, if it is simultaneously both a $P$- and a $Q$-point.
\end{itemize}
Corollary~\ref{main_cor0} implies the following statements.

\begin{corollary} \label{main_cor}
Suppose that  there exists a pseudo-$P_{\mathfrak b^+}$-point.  Then
\\
$\cf^\ast(\Sym(\omega))=\mathfrak b$.
\end{corollary}

\begin{corollary} \label{main_cor2}
 Suppose that $\mathfrak u<\cf^\ast(\Sym(\w))$.
Every two ultrafilters that are not coherent to $Q$-points are
coherent. In particular, if there is no $Q$-point, then  NCF  holds.
\end{corollary}

Corollary~\ref{main_cor} can be compared to
 the following theorem: If $\lambda<\kappa$ are
regular uncountable cardinals such that there exists a simple
$P_\lambda$-point $\U$ and a $P_\kappa$-point $\F$, then
$\cf^\ast(\Sym(\w))\leq\lambda$ (cf. \cite[Theorem~3.4]{To98}).
 The assumption of this theorem
(whose consistency was conjectured in  \cite{BS87}) clearly implies
that $\mathfrak u<\mathfrak s$ and $\U$ is not coherent to $\F$,
 and hence there are exactly two coherence classes of
  ultrafilters (cf.  \cite[Corollary~13]{BM99}).
The question whether there
can  be exactly $n$ coherence classes of  ultrafilters for $1<n<\w$
remains open.

On the other hand, given any ground model of GCH and a regular
cardinal $\nu$ in it, the forcing from \cite{BlaShe89} with
$\delta=\w_1$ and $\nu=\kappa$ ($\delta$ and $\nu$ are the two
parameters there) yields a model of ``there exists a simple
$P_\kappa$-point $\U$ and $\mathfrak b=\w_1\leq 2^\w=\kappa$''.
Combined with Theorem~\ref{main_gr} this gives  the consistency of
the statement ``there exists a simple $P_\kappa$-point $\U$ and
$\w_1=\mathfrak b= \cf^\ast(\Sym(\w)) =\mathfrak b(\U)<\kappa$''.

We shall denote the set of all unbounded nondecreasing elements of
$\w^\w$ by $\w^{\uparrow\w}$. We call a set
$F\subset\w^{\uparrow\w}$ \emph{finitely dominating}, if for every
$x\in\w^\w$ there exists a finite subset $\{f_0,\ldots, f_n\}$ of
$F$ such that $x\leq^\ast\max\{f_0,\ldots, f_n\}$. Following
\cite{MST06} we denote  the minimal size of a family of non-finitely
dominating sets covering $\w^{\uparrow\w}$ by
$\mathrm{cov}(\mathfrak D_{\mathit{fin}})$.

 As the next theorem shows, NCF implies that  $\cf^\ast(\Sym(\w))$ is maximal possible.

\begin{theorem} \label{main1}
$\cf^\ast(\Sym(\w))\geq\mathrm{cov}(\mathfrak D_{\mathit{fin}})$. Moreover,
    NCF implies that \\
 $\cf^\ast(\Sym(\w))=\mathfrak d$.
\end{theorem}

Shelah and Tsaban \cite{SheTsa03} proved that $\max\{\mathfrak
b,\mathfrak g\}\leq \mathrm{cov}(\mathfrak D_{\mathit{fin}}) $, and
the strict inequality is consistent (cf. \cite{MST06}). Thus
Theorem~\ref{main1} improves the lower bound in $\mathfrak
g\leq\cf^\ast(\Sym(\w))$ \cite[Theorem~2.6]{To98}. Combining
Corollary~\ref{main_cor2} and the fact that there are no $Q$-points
under $\mathfrak u<\mathfrak s$ (cf.  \cite[Theorems~13.6.2,
13.8.1]{BZ}
), we get the following:

\begin{corollary}
If $\mathfrak u<\min\{\mathfrak s,\cf^\ast(\Sym(\w))\} $,  then NCF
holds.
\end{corollary}

We do not know whether the inequality $\mathfrak
u<\cf^\ast(\Sym(\w))$ (or even $\mathfrak u<\cf(\Sym(\w))$) implies
NCF. This would be true if $\cf(\Sym(\w))\leq \hot{mcf}=\min\{\hot
b(\F):\F$ is an ultrafilter$\}$ (in particular,  if $\hot{mcf}$ is
attained at some ultrafilter not coherent to a $Q$-point). It
would also be interesting to establish whether NCF implies
$\cf(\Sym(\w))=\mathfrak d$.

This work is a continuation of our
previous paper \cite{BRZ}. We refer the reader to \cite{Va90} for
the definitions and basic properties of small cardinals
which are used but not defined in this paper. All filters are
assumed to be non-principal.

\section{Proofs}

The main technical tool for the
proofs of Theorems~\ref{main_mon} and \ref{main_gr} was developed in
\cite{BRZ}.
This will allow us to  prove some   stronger  technical statements
in this section, namely Propositions~\ref{main_mon_F} and
\ref{main_gr_F}.
 In order to formulate them
we need to recall some definitions.

Let $\F$ be a filter.
Following \cite{BZ06} (our definition of an $[\F]$-cover differs slightly
from the one given in \cite{BRZ,BZ06}, however,
 by  \cite[5.5.2, 5.5.3]{BZ} the two versions are equivalent), we define an
indexed cover $\la B_n : n\in\w \ra$ of a set $X$ to be an {\em
$[\F]$-cover} if there is  an increasing surjection $\phi:\w\to\w$
such that $\phi(\{n\in\w:x\in B_n\})\in\F$ for every $x\in X$.

A subset $X$ of a topological monoid $M$  is defined to be {\em
$[\F]$-Menger} if for every sequence $\la U_n:n\in\w\ra$ of
neighborhoods of 1 in $M$ there is a sequence $\la F_n:n\in\w\ra$ of
finite subsets of $M$ such that $\la U_n\cdot F_n\cap F_n\cdot
U_n:n\in\w\ra$ is an $[\F]$-cover of $X$. The latter happens
if and only if
$$X\subset\bigcup_{F\in\F}\bigcap_{n\in\phi(F)}U_n\cdot F_n\cap F_n\cdot U_n$$
for some monotone surjection $\phi:\w\to\w$.

\begin{definition} \label{de_cfp_F}
For a topological monoid (group) $X$ and a free filter $\F$ on $\w$
by $\cf^{\F}_{m}(X)$ (resp. $\cf^{\F}_{g}(X)$) we denote the minimal
length of  an increasing chain $\la X_i:i<\lambda\ra$ of  proper
submonoids (subgroups) $X_i$ of $X$ such that $
X=\bigcup_{i<\lambda}X_i$ and for every $[\F]$-Menger subset $H$ of
$X$
 there exists $i\in\lambda$ such that $H\subset X_i$.

If no such  chain exists, then we say that $\cf^{\F}_{m}(X)$ (resp.
$\cf^{\F}_{g}(X)$) is undefined.
\end{definition}

It is easy to check that $\cf^\ast_{m}(X)$ (resp. $\cf^\ast_{g}(X)$)
is   $\cf^{\mathfrak Fr}_{m}(X)$ (resp.
$\cf^{\mathfrak Fr}_{g}(X)$), where $\mathfrak Fr$ denotes the
Fr\'echet filter consisting of all cofinite subsets of $\w$.

Let $\F$ be an ultrafilter.
 A sequence $\la b_\alpha:\alpha<\mathfrak b(\F)\ra$ of  increasing
elements of $\w^\w$ is called a \emph{$\mathfrak b(\F)$-scale}, if
it is cofinal with respect to $\leq_\F$ and $b_\alpha\leq_\F
b_\beta$ for all $\alpha\leq\beta<\mathfrak b(\F)$.

 Let us denote  the family of all monotone surjections from $\w$ to $\w$ by $\mathcal S$.
 Following
\cite[\S 10.1]{BZ} (see also \cite{Can88})
we denote for an ultrafilter $\F$  by $\hot q(\F)$ the minimal size
of a subfamily $\Phi$ of $\mathcal S$ such that for every
$\psi\in\mathcal S$ there exists $\phi\in\Phi$ such that $[\phi\leq
\psi]\in\F$. It is clear that there exists a sequence
$\la\phi_\alpha:\alpha<\hot q(\F)\ra\in\mathcal S^{\hot q(\F)}$ such
that $[\phi_\beta<\phi_\alpha]\in\F$ for all $\beta>\alpha$ and for
every $\psi\in\mathcal S$ there exists $\alpha$ with the property
$[\phi_\alpha<\psi]\in\F$. Such a family will be called a
\emph{$\hot q(\F)$-scale}.

Cardinals $\hot b(\F)$ and $\hot q(\F)$ are
 the cofinality and the coinitiality of the linearly ordered set
$(\w^{\uparrow \w},\le_\F)$, which in a certain sense makes them
dual.

If an ultrafilter $\F$ is not
coherent to any $Q$-point  then $\hot b(\F)=\hot q(\F)$,
for a proof see \cite{LafZhu98,Ca89} or  \cite[10.2.5]{BZ}.
 On the other
hand, there
can be ultrafilters $\F$ with $\hot b(\F)\neq\hot q(\F)$, see
\cite{Can88}. As we shall see later, this means that
$\cf^{\F}_{g}(X)$ and $\cf^{\F}_{m}(X)$ are not always well-defined.

\begin{theorem} \label{koreknist}
Let $\F$ be an ultrafilter and
$X$  a first countable $\w$-bounded balanced topological monoid
 (resp. first countable topological group) and suppose that
  one of its finite powers  is not
  a Menger monoid.
\begin{itemize}
\item[$(1)$]  If the cardinal $\cf^{\F}_m(X)$ (resp. $\cf^{\F}_g(X)$) exists, then it is equal to $\hot b(\F)$
and $\hot b(\F)=\hot q(\F)$.
\item[$(2)$]  If $\F$ is not coherent to any $Q$-point, then the cardinal $\cf^{\F}_m(X)$ (resp. $\cf^{\F}_g(X)$) exists and hence
it is equal to $\hot b(\F)=\hot q(\F)$.
\item[$(3)$]  For the group $X=\mathrm{Auth}(\mathbb R_+)$ of the homeomorphisms of the half-line the cardinal $\cf^{\F}_m(X)$ exists
if and only if  $\cf^{\F}_g(X)$ exists if and only if $\F$ is not
coherent to a $Q$-point.
\end{itemize}
\end{theorem}

We postpone the proof of Theorem~\ref{koreknist}
for the moment. It is clear that for a topological group $X$
 the existence of $\cf^{\F}_g(X)$ implies the existence of $\cf^{\F}_m(X)$, and in this case
 $\cf^{\F}_m(X)\le \cf^{\F}_g(X)$.

\begin{question}
 Is the existence of $\cf^{\F}_g(X)$ equivalent to the existence of $\cf^{\F}_m(X)$
(at least for the group $\Sym(\w)$)? Are these cardinals always
equal (if they exist)?
\end{question}

The following result was established in \cite{BRZ}.

\begin{lemma} \label{tt} A topological group (resp. balanced topological monoid) $H$ is $[\LL]$-Menger for some ultrafilter $\LL$ coherent to no $Q$-point if and only if $H$ is algebraically generated by an $[\LL]$-Menger subspace $X\subset H$.
\end{lemma}

The condition in Lemma~\ref{tt} that $\LL$ is not coherent to any
$Q$-point is essential by \cite[Theorem~6.4]{BRZ}. However, we do
not know whether it can be omitted from Theorem~\ref{main_mon},
Theorem~\ref{main_gr} or Corollary~\ref{main_cor0}.

Theorem~\ref{main_mon} is a
special case of  the
following result:

\begin{proposition} \label{main_mon_F}
Let $X$ be a first countable $\w$-bounded balanced topological
monoid such that one of its finite powers is not
a Menger monoid, and let $\F$ be a filter on $\w$. If there exists
an ultrafilter $\LL\supset\F$ that is not coherent to any $Q$-point,
then $\cf^{\F}_{m}(X)$ is well-defined and is
less than or equal to
 $\mathfrak b(\LL)$.
\end{proposition}

\begin{proof}
 Let $\LL\supset\F$ be an
ultrafilter that is not coherent to any $Q$-point,   $\la
b_\alpha:\alpha<\mathfrak b(\LL)\ra$ be a  $\mathfrak b(\LL)$-scale,
and $\la \phi_\alpha:\alpha<\mathfrak q(\LL)=\hot b(\LL)\ra$ be a
$\mathfrak q(\LL)$-scale. Assume that $X^k$ is not a Menger monoid
for some $k\in\w$. Let $\{ U_n: n\in\w\}$ be a local base at the
neutral element $1$ of $X$. Without loss of generality, we may
assume that $U_{n+1}^3\subset U_n$ for all $n\in\w$.
 Applying
\cite[Proposition~7.1]{BRZ}, we can additionally assume that there
exists a sequence
  $\la C_n:n\in\w\ra$  of
countable subsets of $X$ such that $U_n\cdot C_n= C_n\cdot U_n=X$
for all $n$, and for every  $F\in [X]^{<\w}$ there exists $F'\in
[C_n]^{<\w}$   such that $FU_{n+1}\cap U_{n+1}F\subset F' U_n\cap
U_n F'$. Fix an enumeration $\{c_{n,m}:m\in\w\}$ of $C_n$. For a
pair $(\phi,b)\in\mathcal S\times \w^\w$ we set
\begin{eqnarray*}
Y_{\phi,b}= \bigcup_{L\in\LL}\bigcap_{n\in L} U_{\phi(n)}\cdot
\{c_{k,m}:\phi(n)\leq k\leq n,\; m\leq b(n)\}\cap \\
\cap \{c_{k,m}:\phi(n)\leq k\leq n,\; m\leq b(n)\}\cdot U_{\phi(n)}
\end{eqnarray*}
 and denote by
 $X_\alpha$ the submonoid of $X$ generated by
$Y_{\phi_\alpha,b_\alpha}$. A direct verification shows that
$Y_{\phi,b}$ is an $[\LL]$-Menger subset of $X$ for arbitrary pair
$(\phi,b)\in\mathcal S\times\w^\w$ (cf. e.g., the proof of
\cite[Lemma~3.2]{BRZ}), and hence by Lemma~\ref{tt}
 $X_\alpha$  is an $[\LL]$-Menger
submonoid of $X$. Thus $\la X_\alpha:\alpha<\hot{b}(\LL)\ra$ is an
increasing sequence of $[\LL]$-Menger submonoids of $X$. Since $X^k$
is  not a Menger monoid and the $[\LL]$-Menger property is preserved
by finite powers \cite[Corollary~3.5]{BRZ}, each $X_\alpha$ is a
proper submonoid of $X$.

It suffices to show that each $[\F]$-Menger submonoid $H$ of $X$ is
contained in some $X_\alpha$. Given such   $H$ let us find
an increasing $f\in\w^\w$ and $\phi\in\mathcal S$ such that
$$ H\subset \bigcup_{F\in\F}\bigcap_{n\in F} U_{\phi(n)}\cdot \{c_{\phi(n),m}:m\leq f(n)\}\cap \{c_{\phi(n),m}:m\leq f(n)\}\cdot U_{\phi(n)}. $$
(Such $f$ and $\phi$ can be easily constructed by the definition of
the $[\F]$-Menger property.)

Choose $\alpha$ such that $f\leq_{\LL} b_\alpha$ and
$\phi_\alpha\leq_{\LL}\phi$. We claim that $H\subset X_\alpha$.
Indeed, let us fix  $h\in H$ and pick $F_0\in\F$ such that $$h\in
\bigcap_{n\in F_0} U_{\phi(n)}\cdot \{c_{\phi(n),m}:m\leq f(n)\}\cap
\{c_{\phi(n),m}:m\leq f(n)\}\cdot U_{\phi(n)}.$$ Set
$A=[\phi_\alpha\leq\phi]$,  $B=[f\leq b_\alpha]$, and observe that
$A,B  \in\LL$. Then
\begin{eqnarray*}
 h \in \bigcap_{n\in F_0} U_{\phi(n)}\cdot \big\{c_{\phi(n),m}:m\leq f(n)\big\}\cap \big\{c_{\phi(n),m}:m\leq f(n)\big\}
\cdot U_{\phi(n)} \subset\\
\subset \bigcap_{n\in F_0\cap A} U_{\phi_\alpha(n)}\cdot \big\{c_{k,m}:\phi_\alpha(n)\leq k\leq n, m\leq f(n)\big\}\cap \\
\cap \big\{c_{k,m}:\phi_\alpha(n)\leq k\leq n, m\leq f(n)\big\} \cdot U_{\phi_\alpha(n)} \subset \\
 \subset \bigcap_{n\in F_0\cap A\cap B }  U_{\phi_\alpha(n)}\cdot \big\{c_{k,m}:\phi_\alpha(n)\leq k\leq n, m\leq b_\alpha(n)\big\}\cap \\
\cap \big\{c_{k,m}:\phi_\alpha(n)\leq k\leq n, m\leq
b_\alpha(n)\big\} \cdot U_{\phi_\alpha(n)} \subset X_\alpha,
\end{eqnarray*}
which completes  our proof.
\end{proof}

Theorem~\ref{main_gr} is a consequence of  the following:

\begin{proposition} \label{main_gr_F}
Let $G$ be an $\w$-bounded topological group such that one of its
finite powers  is not
a Menger monoid and let $\F$ be a filter on $\w$. If there exists an
ultrafilter $\LL\supset\F$ that is not coherent to any $Q$-point, then
$\cf^{\F}_{g}(G)$ is well-defined and is
less than or equal to
 $\mathfrak b(\LL)$.
\end{proposition}
\begin{proof}
By a  result of Guran \cite{Gu81}, $G$ is topologically
isomorphic
 to a subgroup of a product $\prod_{i\in I} Q_i$, where each $Q_i$
is a second countable group. There exists  $J\in [I]^\w$ with the
property that one of the finite powers of $H:=\mathrm{pr}_J(G)$ is
not a Menger monoid. Indeed, let $k\in\w$ be such that $G^k$ is not
a Menger monoid. There exists a sequence $\la U_n:n\in\w\ra$ of open
neighbourhoods of the neutral element of $G$ such that
$G^k\neq\bigcup_{n\in\w}F_n\:U_n^k\cap U_n^k\:F_n$ for any sequence
$\la F_n:n\in\w\ra$ of finite subsets of $G^k$. Shrinking $U_n$, if
necessary, we may additionally assume that $U_n=\prod_{i\in
J_n}W_{i,n}\times\prod_{i\in I\setminus J_n}Q_i$, where $J_n$ is a
finite subset of $I$ and $W_{i,n}$ is an open neighbourhood of the
neutral element of $Q_i$. Set $J=\bigcup_{n\in\w}J_n$,
$H=\mathrm{pr}_J(G)$, and $V_n=\prod_{i\in
J_n}W_{i,n}\times\prod_{i\in J\setminus J_n}Q_i$. It follows from
the above  that $H^k\neq\bigcup_{n\in\w}K_n\:V_n^k\cap V_n^k\:K_n$
for any sequence $\la K_n:n\in\w\ra$ of finite subsets of $H^k$,
which means that $H^k$ is not a Menger monoid.

 By applying
the same argument as in the proof of Proposition~\ref{main_mon_F}
to the (first countable) group $H$, we
conclude that there exists an appropriate increasing chain $\la
H_\alpha:\alpha<\hot{b}(\LL)\ra$ of proper subgroups of $H$ such
that $H=\bigcup_{\alpha}H_\alpha$. Now $\la
\mathrm{pr}_J^{-1}(H_\alpha):\alpha<\hot{b}(\LL)\ra$ is a
witness for $\cf^{\F}_{g}(G)\leq\hot{b}(\LL)$, which completes our
proof.
\end{proof}
\smallskip

\noindent\textit{Proof of Theorem~\ref{koreknist}.} \
(1) Suppose that $\kappa:=\cf^{\F}_m(X)$ exists and $\kappa<\hot
q(\F)$. All other cases ($\kappa>\hot q(\F)$, $\kappa<\hot b(\F)$,
$\kappa>\hot b(\F)$, or $X$ is a topological group, $\cf^{\F}_g(X)$
exists and $\cf^{\F}_g(X)<\hot q(\F)$, $\cf^{\F}_g(X)>\hot q(\F)$,
$\cf^{\F}_g(X)<\hot b(\F)$, or $\cf^{\F}_g(X)>\hot b(\F)$) are
analogous.

We use the notations from the proof of Proposition~\ref{main_mon_F}.
For every $\alpha<\hot q(\F)$ let
$$
Z_\alpha = \bigcup_{F\in\F}\bigcap_{n\in F} U_{\phi_\alpha(n)}\cdot
\{c_{\phi_\alpha(n),m}: m\leq n\}\cap  \{c_{\phi_\alpha(n),m}: m\leq
n\}\cdot U_{\phi_\alpha(n)}
$$
and observe that $\la Z_\alpha: \alpha<\hot q(\F)\ra$ is an
increasing sequence of $[\F]$-Menger subspaces of $X$ covering $X$.
Let $\la X_\xi:\xi<\kappa\ra$ be a sequence of proper submonoids of
$X$ witnessing for $\cf^{\F}_{m}(X)=\kappa$. Since $\hot q(\F)$ is
regular and for every $\alpha<\hot q(\F)$ there exists $\xi<\kappa$
with $Z_\alpha\subset X_\xi$, we conclude that there exists $\xi$
such that $X_\xi\supset Z_\alpha$ for cofinally many $\alpha\in\hot
q(\F)$, which means $X_\xi=X$ and thus contradicts the assumption
that  $X_\xi$ is a proper submonoid of $X$.
\smallskip

(2) The existence of $\cf^{\F}_m(X)$ (resp. $\cf^{\F}_g(X)$) follows
from Proposition~\ref{main_mon_F} (resp.
Proposition~\ref{main_gr_F}.) The rest is a consequence of the
previous item.
\smallskip

(3) This item follows directly from \cite[Theorem~6.4]{BRZ}. \hfill
$\Box$
\medskip

A sequence $\la U_n : n\in\w\ra$ is called  an \emph{$\w$-cover} of
a set $X$ if for every finite $F\subset X$ there exists $n\in\w$
such that $F\subset U_n$. If, moreover, there exists an increasing
sequence $\la n_k:k\in\w\ra$ of integers such that for every finite
$F\subset X$ and for all but finitely many $k\in\w$ there exists
$n\in [n_k, n_{k+1})$ such that $F\subset U_n$, then the cover  $\la
U_n : n\in\w\ra$ is called \emph{$\w$-groupable}.
\medskip

\noindent\textit{Proof of Corollary~\ref{main}.} In light of
Theorem~\ref{main_gr} it is enough to
verify the following:

\begin{claim}\label{tfr}
 If all
finite powers of an analytic topological group $G$ are
Menger monoids, then $G$ is $\sigma$-bounded\footnote{This fact
 can  be thought of as the analogue for
topological groups of the following result proven in \cite{Ar86}: if
for every sequence $\la u_n:n\in\w\ra$ of open covers of an analytic
space $X$ there exists a sequence $\la v_n:n\in\w\ra$ such that
$v_n\in [u_n]^{<\w}$ and $X=\bigcup_{n\in\w}\cup v_n$, then $X$ is
$\sigma$-compact.}.
\end{claim}
\begin{proof}
Suppose that all finite powers of  $G$ are Menger monoids. By
applying \cite[Lemma~17]{Zd06} and \cite[Prop.~3.1, Lemma~3.2]{BRZ},
we can conclude that $G$ is $[\U]$-Menger for some ultrafilter $\U$.
Given a decreasing base $\la U_n:n\in\w\ra$ at the identity of $G$
we can  find a sequence $\la F_n:n\in\w\ra$ of finite subsets of $G$
such that $\la B_n=F_n U_n\cap U_n F_n:n\in\w\ra$ is an $[\U]$-cover
of $G$.  For every $g\in G$  denote
 the set $\{n\in\w:g\in B_n\}$
 by $\N_g$.

It follows that there exists an increasing number sequence $\la
n_k:k\in\w\ra$ such that $ \bigcup_{\N_g\cap [n_k,
n_{k+1})\neq\emptyset} [n_k, n_{k+1})\in\U  $ for all $g\in G$ (if
$\phi$ is a finite-to-one surjection witnessing for $\la
B_n:n\in\w\ra$ being an $[\U]$-cover, then the sequence $\la
\min\phi^{-1}(k)\ra_{k\in\w}$ is as required.) Let $F'_k$ be a
finite subset of $G$ such that $ D_k:= U_k F'_k\cap F'_k U_k\supset
\bigcup_{n\in [n_k, n_{k+1})} B_n. $ $\la D_k:k\in\w\ra$ is clearly
an $\w$-cover of $G$. Applying \cite[Theorem~4.5]{Sa06} (see also
\cite[Theorem~7]{Zd05}), we conclude that $\la D_k:k\in\w\ra$ is
$\w$-groupable.

Let $\la k_m:m\in\w\ra$ be an increasing number sequence witnessing
for this. Set $Y_m=\bigcap_{l\geq m}\bigcup_{k\in [k_m, k_{m+1})}
D_k$. A direct verification shows that  each $Y_m$ is totally
bounded and $G=\bigcup_{m\in\w} Y_m$. \hfill $\Box$
\end{proof}

\textit{Proof of Corollary~\ref{main_cor}.} Suppose that  $\U$ is a
pseudo-$P_{\mathfrak b^+}$-point. Since  $\phi(\U)$ is clearly a
pseudo-$P_{\mathfrak b^+}$-point for every finite-to-one $\phi$,
$\U$ is not coherent to a $Q$-point by  \cite[Theorem~13.8.1]{BZ}.
Therefore $\cf^\ast(\Sym(\kappa))\leq\mathfrak b(\U)$. It suffices
to apply the following result of Nyikos \cite{Ny84} (see
\cite[Proposition~5]{BM99} or \cite[Theorem~13.2.1,
Corollary~10.3.2]{BZ} for its proof): \textit{If  $\LL$ is
pseudo-$P_{\mathfrak b^+}$-point, then  $\mathfrak b(\LL)=\mathfrak
b$}. \hfill $\Box$
\medskip

\textit{Proof of Corollary~\ref{main_cor2}.} Let $\U$ be an
ultrafilter generated by  $\mathfrak u$ many subsets of $\w$. It is
well-known that  $\mathfrak b(\U)=\mathfrak d$ and $\U$ is coherent
to any ultrafilter $\F$ such that $\mathfrak b(\F)>\mathfrak u$, see
\cite[Theorem~10.3.1]{BZ} or \cite[Theorem~12]{BM99}. It suffices to
apply  Corollary~\ref{main} and
the transitivity of the coherence relation.
\hfill $\Box$
\medskip

\begin{lemma} \label{zau1}
 If $F\subset\w^\w$ is a finitely dominating family of strictly
increasing functions, then $\bigcup_{f\in F}S_f$ generates
$\Sym(\w)$.
\end{lemma}
\begin{proof}
Let $H=\la \bigcup_{f\in F}S_f\ra$ and
  $\pi\in\Sym(\w)$ be such that all its orbits
are finite, i.e. for every $n\in\w$ the set $\{\pi^k(n):k\in\w\}$ is
finite,
 where $\pi^1=\pi$ and $\pi^{k+1}=\pi\circ\pi^k$.
Let $\A=\{a_i:i\in\w\}$ be the enumeration of orbits of $\pi$ such
that $\min a_i<\min a_{i+1}$ for all $i$. The following claim is
obvious.

\begin{claim} \label{cl1}
 There exist two increasing sequences $\la n^0_i:i\in\w  \ra$ and $\la n^1_i:i\in\w\ra$
of natural numbers such that for every $a\in\A$ there exists a pair
$\la i,j\ra\in\w\times 2$ such that $a\subset [n^j_i,n^j_{i+1})$.
\end{claim}

Let $h\in\w^\w$ be an increasing function such that
$h(n^j_i)\geq\max\{\pi(m),\pi^{-1}(m):m\in [n^j_i,n^j_{i+1})\}$ for
all $i$ and $j$, and $F_0$ be a finite subset of $F$ such that
$h\leq^\ast\max F_0$. Fix any $a\in\A$ and find $\la
i,j\ra\in\w\times 2$ such that $a\subset [n^j_i,n^j_{i+1})$. Let
$f\in F_0$ be such that $f(n^j_i)>h(n^j_i)$. By the definition of
$h$ the above implies $\pi(m),\pi^{-1}(m)\leq h(n^j_i)\leq
f(n^j_i)\leq f(m)$ for every $m\in a$. Therefore for every $a\in \A$
there exists $f_a\in F_0$ such that $\pi(m),\pi^{-1}(m)<f_a(m)$ for
all $m\in a$. Set $\pi_f=\pi\uhr\bigcup\{a\in\A:f_a=f\}$ and note
that $\pi_f\in S_f$ and $\pi=\circ_{f\in F_0}\pi_f$ (the latter
composition obviously does not depend on the order in which we take
$\pi_f$'s). Hence $\pi\in H$.

$\Sym(\w)$ is easily seen to be a $G_\delta$-subset of $\w^\w$.
Therefore $\Sym(\w)$ with the topology $\tau$ inherited from $\w^\w$
is a Polish topological group. It is also easy to check that the set
$E$ of all permutations of $\w$ with finite orbits is  a dense
$G_\delta$ of $(\Sym(\w),\tau)$, and hence $E\circ E\supset
\Sym(\w)$ by the Baire Category Theorem. It suffices to note that $
E\circ E\subset H$.
\end{proof}

\medskip

\textit{Proof of Theorem~\ref{main1}.} The first statement is a
direct consequence of Lemma~\ref{zau1}: Suppose that
$\kappa=\cf^\ast(\Sym(\w))<\mathrm{cov}(\mathfrak D_{\mathit{fin}})$
and $\la G_\alpha:\alpha<\kappa\ra$ is an increasing sequence of
proper subgroups of $\Sym(\w)$ witnessing for that. Set
$B_\alpha=\{f\in\w^{\uparrow\w}: S_f\subset G_\alpha\}$. By the
definition of $\cf^\ast(\Sym(\w))$,
$\bigcup_{\alpha<\kappa}B_\alpha=\w^{\uparrow\w}$. Since
$\kappa<\mathrm{cov}(\mathfrak D_{\mathit{fin}})$, there exists
$\alpha<\kappa$ such that $B_\alpha$ is finitely dominating, which
by  Lemma~\ref{zau1} implies that $G_\alpha=\Sym(\w)$ and hence
contradicts the properness of $G_\alpha$.

The second one follows from the fact that
 NCF implies that $\mathrm{cov}(\mathfrak
D_{\mathit{fin}})=\mathfrak d$. Indeed, suppose that NCF holds. Then
  $\mathfrak b(\F)=\mathfrak d$ for all ultrafilters $\F$,
see e.g. \cite[Theorem~16]{Bl86} or \cite[12.3.1]{BZ}. In addition,
every not finitely dominating subset of $\w^{\uparrow\w}$ is
$\leq_\F$-bounded for every ultrafilter $\F$.
 \hfill
$\Box$

\section{Appendix}

Following the suggestion  of the referee, we include here from
\cite{BZ} an essentially self-contained proof of the fact that there
are no $Q$-points (in fact, rare ultrafilters) provided that
$\mathfrak r<\mathfrak s$. This is a direct consequence of
Corollary~\ref{col9_2_5} and Proposition~\ref{13.8.1} below.

The easiest way to do this   would be to simply  copy relevant
pieces of \cite{BZ}. But since the book \cite{BZ} is available
online, this  does not make much sense. Therefore we take another
approach and present a simplified proof. The simplification comes
mainly from the obvious equality $\F=\F^\perp$ which holds for all
ultrafilters. However, this simplification seems to hide  some
ideas.

In what follows $\mathfrak{F}r$ denotes the filter of cofinite
subsets of $\w$. By a \emph{semifilter} we mean a subset $\mathcal
S$ of $[\w]^\w$ which is closed with respect to taking supersets of
its elements and such that $S\cap A\in\mathcal S$ for all
$S\in\mathcal S$ and $A\in \mathfrak{F}r$. For a subset $\Psi$ of
$\w\times \w$ and $n\in\w$ we set $\Psi(n)= \{m\in\w:
(n,m)\in\Psi\}$ and $\Psi^{-1}(n)= \{m\in\w: (m,n)\in\Psi\}$.
$\Psi\subset\w\times\w$ is called a \emph{finite-to-finite
multifunction}, if $\Psi(n),\Psi^{-1}(n)$ are finite and nonempty
for all $n\in\w$. The family of all finite-to-finite multifunction
will usually be considered with the preorder $\subset^*$. A
semifilter $\mathcal S_0$ is said to be \emph{subcoherent} to a
semifilter $\mathcal S_1$, if there exists a finite-to-finite
multifunction $\Psi$ such that $\Psi(\mathcal S_0)\subset \mathcal
S_1$, where $\Psi(\mathcal S_0)=\{\Psi(S):S\in\mathcal S_0\}$ and
$\Psi(X)=\bigcup_{n\in X}\Psi(n)$ for all $X\subset\w$. Semifilters
$\mathcal S_0$ and $\mathcal S_1$ are called \emph{coherent}, if
each of them is subcoherent to the other  one. A direct verification
shows that the subcoherence relation is an equivalence relation. The
equivalence class of a semifilter $\mathcal S$ will be denoted by
$[\mathcal S]$. Each family $\B$ of infinite subsets of $\w$
generates a semifilter, namely the smallest semifilter $\la\B\ra$
containing $\B$\footnote{Note that in this appendix  the notation
$\la\cdot\ra$ has a different  meaning than in the main part of the
paper.}. Given a semifilter $\mathcal S$, we denote  by
$\mathrm{non}[\mathcal S]$  the smallest size of a  family
$\B\subset [\w]^\w$ such that $\la \B\ra$ is not subcoherent to
$\mathcal S$. For an ultrafilter $\F$ we denote by
$\mathrm{cov}[\F]$ the minimal size of a family $\mathsf S\subset
[\F]$ such that $\cap\mathsf S=\mathfrak{F}r$. The increasing
sequence of natural numbers whose range coincides with an infinite
subset $X$ of $\w$ will be denoted by $e_X$. An ultrafilter $\U$ is
called  \emph{rare} if the collection $\{e_F:F\in\F\}$ is
dominating. It is clear that every $Q$-point is rare and the
question whether the existence of a rare ultrafilter implies the
existence of a $Q$-point is open.

 The proof of the following statement is fairly
simple and can be found in the introductory part of \cite{BZ}.

\begin{proposition} \label{basic_facts}
\begin{enumerate}
\item For every finite-to-finite multifunction $\Psi$
there exists an increasing sequence $\la n_k:k\in\w\ra$ of natural
numbers with $n_0=0$ such that $\Psi(n)\subset [n_{k-1},n_{k+2})$
for all $n\in [n_k, n_{k+1})$. Therefore the cofinality of the
family of all finite-to-finite multifunctions equals $\mathfrak d$
and any family of finite-to-finite multifunctions of size
$<\mathfrak b$ has an upper bound.

\item $\mathrm{cov}[\F]\geq \mathfrak b$ and $\mathrm{non}[\F]\leq\mathfrak d$
for all ultrafilters $\F$.
\item Let $\mathcal S$ be a semifilter and $\F$ be a ultafilter.
Then $\mathcal S$ is subcoherent (resp. coherent) to $\F$ if and
only if there exists a monotone surjection $\psi:\w\to\w$ such that
$\psi(\mathcal S)\subset\psi (\F)$ (resp. $\psi(\mathcal S)=\psi
(\F)$).
\item The restriction to ultrafilters of the coherence relation on
the set of all semifilters coincides with the near coherence
relation on ultrafilters (see the definition after
Proposition~\ref{notcoh_q}.)
\end{enumerate}
\end{proposition}

The following statement is a special case of
\cite[Theorem~9.2.5]{BZ}.

\begin{proposition}\label{9_2_5}
Suppose that $\F$ is an ultrafilter, $\mathsf C\subset [\F]$,
$|\mathsf C|<\mathrm{cov}[\F]$. Then for every family $\B\subset
[\w]^\w$ of size less than $\mathrm{cov}[\F]$ there exists a
monotone surjection $\psi:\w\to\w$ such that
$\psi(\B)\subset\psi(\bigcap\mathsf C)$.
\end{proposition}
\begin{proof}
 For every $B\in\B$ and $\C\in\mathsf C$ we denote by $\C_B$ the semifilter consisting
of all infinite subsets $X$ of $\w$ such that
$$ \exists C\in\C\: \forall a,b\in\w\: (a,b\in\w\setminus X \wedge [a,b)\cap C\neq\emptyset \to [a,b)\cap B\neq\emptyset).   $$
Given an arbitrary $B\in\B$, consider the finite-to-finite
multifunction $\Psi_B:\w\Rightarrow\w$ assigning to each $n\in\w$
the interval $\Psi_B(n)=[n,\min (B\setminus[0,n))]$. Observe that
$\Psi_B(\C)\subset\C_B$ for all $\C\in \mathsf C$. Indeed, suppose
that $a,b\in\w\setminus\Psi_B(C)$ for some $C\in\C$ and $[a,b)\cap
C\neq\emptyset$. The inclusion $a\in\w\setminus\Psi_B(C)$ means that
$a\not\in C$ and $a>\min (B\setminus[0,n))$ for all $n<a$ with $n\in
C$. Similarly for $b$. Let $m\in C\cap [a,b)$. It follows from the
above that $\min (B\setminus[0,m))<b$, and hence $[a,b)\cap
B\neq\emptyset$. Therefore $C\in\C$ is a witness for $\Psi_B(C)$
being an element of $\C_B$.

Observe that the semifilter $\la\Psi_B(\C)\ra$ belongs to $[\F]$.
Since $|\B|,|\mathsf C|<\mathrm{cov}[\F]$, the intersection
$\bigcap\{\la\Psi_B(\C)\ra:B\in\B,\C\in\mathsf C\}$ contains a
co-infinite set $X$. Let $\la n_k:k\in\w\ra$ be an increasing
enumeration of $\w\setminus X$ and $\psi^{-1}(k)=[n_k,n_{k+1})$. We
claim that $\psi(\B)\subset\psi(\cap\mathsf C)$. Indeed, let us fix
$B\in\B$ and $\mathcal C\in\mathsf C$. Since
$X\in\la\Psi_B(\C)\ra\subset\C_B$, there exists $C\in C$ such that
$$\forall a,b\in\w\: (a,b\in\w\setminus X \wedge [a,b)\cap C\neq\emptyset \to [a,b)\cap B\neq\emptyset), $$
which means that $\psi(C)\subset\psi (B)$, and hence
$\psi(B)\in\psi(\C)$. Since $B$ and $\C$ are  arbitrary elements of
$\B$ and $\mathsf C$, respectively, our proof is completed.
\end{proof}

\begin{corollary} \label{col9_2_5}
Let $\F$ be an ultrafilter. Then
$\mathrm{non}[\F]\geq\mathrm{cov}[\F]$.
\end{corollary}

The following proposition is a special case of
\cite[Theorem~13.8.1]{BZ}.

\begin{proposition} \label{13.8.1}
Let $\F$ be a rare ultrafilter. Then
\begin{enumerate}
 \item $\mathrm{non}[\F]\leq \mathfrak r$; and
\item $\mathrm{cov}[\F]\geq\mathfrak s$.
\end{enumerate}
\end{proposition}
\begin{proof}
1. By the inequality $\mathrm{non}[\F]\leq \mathfrak d$ we may
assume $\mathfrak r<\mathfrak d$. Since $\F$ is rare, so is
$\psi(\F)$ for any monotone surjection $\psi:\w\to\w$. Applying
Proposition~\ref{basic_facts}(3) we conclude that   no semifilter
${\mathcal S}\in [\F]$ can be generated by fewer than $\mathfrak d$
sets. Let $\U$ be an ultrafilter with $\U\subset\la\B\ra$ for some
$\B\subset [\w]^\w$ with $|\B|=\mathfrak r$. It follows from the
above
 that
$\U\not\in [\F]$,  hence $\U$ is not subcoherent to $\F$, and
consequently $\la B \ra$ is neither subcoherent to $\F$. This yields
$\mathrm{non}[\F]\leq |\B|=\mathfrak r$.

2. First we  show that there exists a subfamily $\B\subset\F$ of
size $|\B|=\mathfrak b$ without an infinite pseudointersection.
Indeed, let $\la f_\alpha:\alpha<\mathfrak b\ra$ be a $\mathfrak
b$-scale, i.e. an increasing and unbounded with respect to $\leq^*$
sequence. Since $\F$ is rare, for every $\alpha$ there exists
$F_\alpha\in\F$ such that $e_{F_\alpha}\geq^* f_\alpha$. If $X\in
[\w]^\w$ is such that $X\subset^*F_\alpha$ and
$F_\alpha\not\subset^* X$, then $e_X\geq^* f_\alpha$, and hence the
existence of an infinite pseudointersection of $\la
F_\alpha:\alpha<\mathfrak b\ra$ would contradict the unboundedness
of $\la f_\alpha:\alpha<\mathfrak b\ra$.

Thus for every semifilter ${\mathcal S}\in [\F]$ there exists a
subfamily ${\mathcal S}'\in [{\mathcal S}]^{\mathfrak b}$ without an
infinite pseudointersection.

Since $\mathrm{cov}[\F]\geq \mathfrak b$, we can assume that
$\mathfrak s>\mathfrak b$. We proceed in the same way as in
\cite[Theorem~9.2.7(7)]{BZ}. Set $\lambda=\mathrm{cov}[\F]$ and find
a family  $\mathsf S\subset[\F]$ such that $|\mathsf S|=\lambda$ and
$\cap\mathsf S=\mathfrak{F}r$. For every ${\mathcal S}\in\mathsf S$
find $\B_{\mathcal S}\subset{\mathcal S}$ of size $|\B_{\mathcal
S}|=\mathfrak b$ such that $\B_{\mathcal S}$ has no infinite
pseudointersection. It suffices to prove that $\bigcup\{\B_{\mathcal
S}:{\mathcal S}\in\mathsf S\}$ is a splitting family. Indeed, let us
fix $X\in [\w]^\w$. Since $\w\setminus X\not\in\mathfrak{F}r$, there
exists ${\mathcal S}\in\mathsf S$ such that $\w\setminus
X\not\in{\mathcal S} $, and hence $B\not\subset^* \w\setminus X$ for
all $B\in\B_{\mathcal S}$. In other words, all elements of
$\B_{\mathcal S}$ have infinite intersection with $X$. If none of
the elements of $\B_{\mathcal S}$ splits $X$, we get that
$X\subset^* B$ for all $B\in\B_{\mathcal S}$, which contradicts our
choice of $\B_{\mathcal S}$. Therefore $X$ is split by some element
of $\B_{\mathcal S}$, and hence $\bigcup\{\B_{\mathcal S}:{\mathcal
S}\in\mathsf S\}$ is a splitting family, which completes our proof.
\end{proof}


\begin{thebibliography}{MM}


\bibitem{Ar86} Arkhangel'ski\u{\i}, A., {\it Hurewicz spaces, analytic sets and fan tightness of function spaces,}
Dokl. Akad. Nauk SSSR  \textbf{287}  (1986),   525--528. (In Russian)

\bibitem{BRZ} Banakh, T.; Repov\v{s}, D.; Zdomskyy, L.,
{\it $o$-Boundedness of free topological groups},
Topology Appl. \textbf{157} (2010), 466--481.

\bibitem{BZ}  Banakh, T.; Zdomskyy, L., {\it Coherence of Semifilters}, book in preparation.\\
\texttt{http://www.franko.lviv.ua/faculty/mechmat/Departments/Topology/booksite.html}.

\bibitem{BZ06} Banakh, T.; Zdomskyy, L.,  {\it Selection principles and infinite games on multicovered
 spaces,}  in: Selection Principles and Covering Properties
in Topology
(Lj. D.R. Ko\' cinac, ed.),
 Quaderni  di Matematica \textbf{18},
Dept. Math., Seconda Universita di Napoli, Caserta (2006), 1--51.

\bibitem{Bl86} Blass, A., {\it Near coherence of filters. I. Cofinal equivalence of models of arithmetic,}
  Notre Dame J. Form. Log.  \textbf{27}  (1986),  579--591.

\bibitem{BM99} Blass, A.; Mildenberger, H.,
{\it On the cofinality of ultrapowers,}
   J. Symbolic Logic  \textbf{64}  (1999),   727--736.


\bibitem{BS87} Blass, A.; Shelah, S.;
{\it  There may be simple $P_{\aleph_1}$- and $P_{\aleph_2}$-points
and the Rudin-Keisler ordering may be downward directed,}
 Ann. Pure Appl. Logic \textbf{33}  (1987),   213--243.

\bibitem{BlaShe89} Blass, A.; Shelah, S.,
{\it Ultrafilters with small generating sets,}
  Israel J. Math.  \textbf{65}  (1989),   259--271.



 \bibitem{Can88} Canjar, R.M, {\it Countable ultraproducts without CH,}
 Ann. Pure Appl. Logic  \textbf{37}  (1988), 1--79.

\bibitem{Ca89} Canjar, R.M., {\it Cofinalities of countable ultraproducts: the existence
theorem,}   Notre Dame J. Form. Log. \textbf{30} (1989), 539--542.


\bibitem{Gu81} Guran, I., {\it Topological groups similar to Lindel\"of groups},   Dokl. Akad. Nauk SSSR  \textbf{256}  (1981), 1305--1307. (In Russian)


\bibitem{LafZhu98}  Laflamme, C.; Zhu, J.-P.,
{\it The Rudin-Blass ordering of ultrafilters,}
J. Symbolic Logic  \textbf{63}  (1998),   584--592.

\bibitem{MN90} Macpherson, H.D.; Neumann, P.M., {\it Subgroups of infinite symmetric groups},
   J. Lond. Math. Soc. (2)  \textbf{42}  (1990), p. 64--84.

\bibitem{MST06} Mildenberger, H.; Shelah, S.;  Tsaban, B., {\it Covering the Baire
space with meager sets,}  Ann. Pure Appl. Logic \textbf{140} (2006),
60--71.



\bibitem{Ny84} Nyikos, P.,
{\it Special ultrafilters and cofinal subsets of $\w$}, preprint,
1984.


\bibitem{Sa06} Sakai, M., {\it Two properties of $C\sb p(X)$ weaker than the
Fr\'{e}chet Urysohn property,}  Topology Appl.  \textbf{153}  (2006),  2795--2804.

\bibitem{SheTsa03}
Shelah, S.; Tsaban, B.,
{\it Critical cardinalities and additivity
properties of combinatorial notions of smallness,}
J. Appl. Anal.  \textbf{9}  (2003),   149--162.

\bibitem{To98} Thomas, S., {\it Groupwise density and the cofinality of the infinite
symmetric group},   Arch. Math. Logic \textbf{37} (1998), 483--493.


\bibitem{Va90} Vaughan, J., {\it Small uncountable cardinals and topology}, in:
  Open problems in topology (eds. J.~van Mill, G.M.~Reed),
Elsevier Sci. Publ., Amsterdam 1990,  197--216.

\bibitem{Zd05} Zdomskyy, L., {\it A semifilter approach to selection principles,}
 Comment. Math. Univ. Carolin.  \textbf{46}  (2005),  525--539.

\bibitem{Zd06} Zdomskyy, L., \emph{o-Boundedness of free objects over
 a Tychonoff space},   Mat. Stud. \textbf{25} (2006), 10--28.

\end{thebibliography}
\end{document}